\documentclass[11pt,a4paper]{amsart}
\usepackage{texmac}
\operators{Stab,Prim,Inn,Out,Inf}

\bbsymbols{b}{C,F,Z,Q}
\calsymbols{c}{H,U,X,Y,P,M,O}

\begin{document}

\title[Simplicity of twists of abelian varieties]{Simplicity of twists of abelian varieties}
\author{Alex Bartel}
\thanks{This research is partly supported by a research fellowship from the
Royal Commission for the Exhibition of 1851.}
\date{}
\address{Department of Mathematics, Warwick University,
Coventry CV4 7AL, UK}
\email{a.bartel@warwick.ac.uk}
\llap{.\hskip 10cm} \vskip -0.8cm
\maketitle

\begin{abstract}
We give some easy necessary and sufficient criteria for twists of abelian
varieties by Artin representations to be simple.
\end{abstract}

\section{Introduction}\label{sec:intro}
Let $A/k$ be an abelian variety over a field, let $R\leq \End(A)$ be
a commutative ring of endomorphisms of $A$ (here and in the sequel, we regard
the abelian varieties as schemes \emph{over a base}, and this is also the
category in which our morphisms will live; in particular, $\End(A)$ denotes
endomorphisms of $A$ defined over $k$; the same remark applies to statements
like ``$A$ is principally polarised'', etc.), and let $K/k$ be a finite Galois
extension with Galois group $G$. Let $\Gamma$ be an $R[G]$-module, together
with an isomorphism $\psi: R^n\rightarrow\Gamma$ for some $n$.
Attached to this data is the so-called twist of
$A$ by $\Gamma$, denoted by $B=\Gamma\otimes_R A$, which is an abelian
variety over $k$ with the property that the base change $B_K = B\times_k K$
is isomorphic to $(A_K)^n$.

As soon as $n>1$, $B$ is, by its very definition, never absolutely simple.
But it can be simple over $k$, and to know when this is the case is important
for some applications, see e.g. \cite{Howe}. If $A'$ is a proper abelian
subvariety of $A$, then $\Gamma\otimes_R A'$ is a proper abelian subvariety of
$\Gamma\otimes_R A$. Similarly, if $\Gamma'\leq \Gamma$ is an $R$-free
$R[G]$-submodule of strictly smaller $R$-rank, then $\Gamma'\otimes_R A$ is
isogenous to a
proper abelian subvariety of $\Gamma\otimes_R A$. The purpose of this
note is to point out that, under some mild additional hypotheses (and in
particular over number fields in the generic case, when $\End_{\bar{k}}(A)\cong
\Z$), these are the only two ways in which $B$ can fail to be simple.

As a concrete example, we mention the following generalisation of Howe's analysis
\cite{Howe}:
\begin{theorem}\label{thm:1}
Let $A/k$ be a simple abelian variety of
dimension 1 or 2 over a number field, let $p$ be an odd prime
number and let $K/k$ be a Galois extension with Galois group $G$ of order $p$.
If $A$ is not absolutely simple or not principally polarised, assume that $p>3$.
Let $I$ be the augmentation ideal in $\Z[G]$, i.e. the kernel of the map
$\Z[G]\rightarrow \Z$, $g\mapsto 1$ $\forall g\in G$. Then
$I\otimes_{\Z} A$ is simple if and only if $\End(A)\otimes \Q$ does not
contain the quadratic subfield of $\Q(\mu_p)$.
\end{theorem}
\begin{remark}
If $p=2$, then $I\otimes_{\Z} A$ is a quadratic twist of $A$, and so also simple
if $A$ is. Since, for all $p$,
$I\otimes \Q$ is the unique non-trivial irreducible $\Q[G]$-module, the
theorem completely deals with simplicity of those twists of elliptic curves
and of principally polarised absolutely simple abelian surfaces
that are trivialised by a cyclic prime degree extension.
\end{remark}
\begin{remark}
By computing the endomorphism ring of $I\otimes \Q$ as a $\Q[G]$-module,
Howe \cite{Howe} showed part of one implication in the case $\dim(A)=1$: he
proved that if $E/k$ is a non-CM elliptic curve, then $I\otimes_\Z E$ is
simple. In the proof of the theorem that we present,
one does not need to know the endomorphism ring of $I\otimes \Q$ to deduce the
result for elliptic curves; one does, however, need to know it to prove the
statement for abelian surfaces.
\end{remark}

The same technique yields uniform statements for higher dimensional abelian varieties,
where the restriction on $p$ depends on the dimension of the variety:
\begin{theorem}\label{thm:2}
Fix an integer $d$. There exists an integer $p_0$ such that for all number fields
$k$, all simple abelian varieties $A/k$ of dimension $d$, all primes $p>p_0$,
and all Galois extensions $K/k$ with cyclic Galois group $G$ of order $p$,
the twist $I\otimes_{\Z}A$ is simple if and only if $\End(A)\otimes\Q$ does not
contain a subfield of $\Q(\mu_p)$ other than $\Q$. Here, $I$ is, as in Theorem
\ref{thm:1}, the augmentation ideal in $\Z[G]$.
\end{theorem}

Similarly concrete results can be obtained for twists by other representations,
and we give several more examples in the same vein in the last section.

\begin{acknowledgements}
I would like to thank Barinder Banwait for bringing Howe's paper to my
attention, which motivated this work. Many thanks are due to Victor Rotger
for very helpful email correspondence. I gratefully acknowledge the financial
support by the Royal Commission for the Exhibition of 1851.
\end{acknowledgements}

\section{Endomorphisms of twists of abelian varieties}
In this section we begin by recalling (see \cite[\S III.1.3]{Serre}) the
definition of a twist of an abelian variety by an Artin representation, and
then give sufficient conditions for the endomorphism ring of such a twist to
be an integral domain, equivalently for the twist to be simple. We strongly
recommend \cite{MRS} for a very thorough
treatment of twists of abelian varieties, and, more generally, of commutative
algebraic groups.

Let $Y/k$ be an abelian variety, and $K/k$ a finite Galois extension with Galois
group $G$.
A $K/k$-form of $Y$ is a pair $(X,f)$, where $X/k$ is an abelian variety, and
$f:Y_K\rightarrow X_K$ is an isomorphism, defined over $K$. There is an obvious
notion of isomorphism between such pairs, and the set of isomorphism classes
of $K/k$-forms of $Y$ is in bijection with the pointed set $H^1(G,\Aut Y_K)$,
where the $G$-action on $\Aut_K Y$ is given by\footnote{we
adhere to the common convention that the superscript for the action is written on the
right, even though this is actually a left action}
$\phi^\sigma = \sigma\circ\phi\circ\sigma^{-1}$
for $\sigma\in G$ and $\phi\in\Aut(Y_K)$. The bijection is given by assigning to
a $K/k$-form $(X,f)$ the cocycle represented by $\sigma\mapsto f^{-1}f^\sigma$,
where, as before, $f^\sigma$ is defined to be $\sigma\circ f\circ\sigma^{-1}$.

Now, suppose that $A/k$ is an abelian variety, and $R\leq \End(A)$
a commutative ring. With $K/k$ and $G$ as above, let $\Gamma$ be an
$R[G]$-module, together with an $R$-module isomorphism
$\psi: R^n\rightarrow\Gamma$ for some $n\in\N$. Then the map
$a_\Gamma:\sigma\mapsto \psi^{-1}\psi^{\sigma}=\psi^{-1}\circ \sigma \circ\psi\in
\GL_n(R)\leq \Aut_KA^n$ defines a cocycle in $H^1(G,\Aut (A_K)^n)$. Indeed,
note that since $G$ acts trivially on automorphisms of $A^n$ that are defined
over $k$, as is the case for $\GL_n(R)\leq \Aut (A_K)^n$, 1-cocycles whose image
lies in $\GL_n(R)$ are simply group homomorphisms. The twist $B$ of $A$ by
$\Gamma$, written $B=\Gamma\otimes_R A$ is, by definition, the $K/k$-form of $A^n$
corresponding to the cocycle $a_\Gamma$.

We now come to the endomorphism ring of $B$. Our aim is to find criteria for
$B$ to be simple, equivalently for $\End(B)$ to be a division ring.
In theory, one can easily describe
$\End(B)$ in terms of the $G$-module structure of $\End_K(A)$ and $\End_{R}(\Gamma)$:
\begin{lemma}\label{lem:MRS}
There is an isomorphism
$$
\End(\Gamma\otimes_R A) \stackrel{\sim}{\rightarrow} (\End_{R}(\Gamma)\otimes \End_K(A))^G.
$$
\end{lemma}
\begin{proof}
This immediately follows from \cite[Proposition 1.6]{MRS}, by noting that
the absolute Galois group of $k$ acts on $\Gamma$
through the quotient $G$.
\end{proof}

However, in the most general form, this description is not easy to use for
determining when the right hand side of the equation is a division ring.
On the other hand, generically the situation is much better.

\begin{assumption}\label{ass}
For the rest of this section, assume that $\End(A) = \End(A_K)$. Since we are
interested in criteria for $B$ to be simple, we will
also assume from now on that $A$ itself is simple, therefore so is $A_K$
by the previous assumption.
\end{assumption}
\begin{remark}
This assumption is generically satisfied over number fields in the following
sense: fix an abelian variety $A$ over a number field $k$,
and a Galois group $G$. A result of Ribet and Silverberg \cite{Silverberg,Ribet}
says that, given any subring
$\cO\subseteq \End_{\bar{k}}(A)$ there exists a unique minimal extension
$L_{\cO}/k$ such that $\cO\subseteq \End_{L_{\cO}}(A)$. So $\End_K(A)=\End(A)$
whenever $K\cap L_S=k$.
\end{remark}

\begin{notation}
The following notation will be retained throughout the paper:
\begin{itemize}
\item $K/k$ --- a Galois extension of fields with Galois group $G$;
\item $A/k$ --- a simple abelian variety;
\item $S=\End(A)$;
\item $R\leq S$ --- a commutative subring;
\item $\Gamma$ --- an $R$-free $R[G]$-module;
\item $B=\Gamma\otimes_R A$ --- the twist of $A$ by $\Gamma$, which is an abelian variety over $k$;
\item $D=S\otimes_\Z \Q$ --- a division algebra;
\item $F=R\otimes_\Z \Q$ --- a field contained in $D$;

\end{itemize}
\end{notation}

Under Assumption \ref{ass}, Lemma \ref{lem:MRS} becomes
\begin{eqnarray}\label{eqn}
\End(B) \cong \End_{R[G]}(\Gamma)\otimes_R S.
\end{eqnarray}

In general, it is a subtle question with a rich literature when the tensor
product of two division rings over a common subring is a division ring.
But for a generic polarised abelian variety, $S=\Z$. More generally, if $S$ is
commutative, Schur's Lemma furnishes an elementary answer to the question of
simplicity of $B$:

\begin{proposition}\label{prop:main}
Assume, in addition to Assumption \ref{ass},
that $S$ is commutative, i.e. that $D$ is a field.
Then $B$ is simple if and only if $\Gamma\otimes_R D$
is a simple $D[G]$-module.
\end{proposition}
\begin{proof}
The twist $B$ is simple if and only if
$\End(B)$ is a division ring, if and only if
$$
\End(B)\otimes_{\Z}\Q\cong \End_{R[G]}(\Gamma)\otimes_R D
$$
is a division algebra. 
It is an elementary computation that when $S$ is commutative,
$\End_{R[G]}(\Gamma)\otimes_R D$ is precisely the endomorphism ring of the
$D[G]$-module $\Gamma\otimes_RD$, the isomorphism given by

\begin{eqnarray*}
\End_{R[G]}(\Gamma)\otimes_R D & \rightarrow & \End_{D[G]}(\Gamma\otimes_R D),\\
\alpha\otimes f & \mapsto & (\gamma\otimes g \mapsto \alpha(\gamma)\otimes fg).
\end{eqnarray*}
We deduce that, by Schur's Lemma, $B$ is simple if and only
if $\Gamma\otimes_R D$ is a simple $D[G]$-module.
\end{proof}

There is slightly different way of phrasing this discussion, which is
closer to Howe's original proof.
Since $A_K$ is assumed to be simple, $S$ is a division ring, and
$\End_K(A^n)\cong M_n(S)$, the $n$-by-$n$ matrix ring over $S$.
Since the base change of $B$ to $K$ is isomorphic to
$(A_K)^n$, any endomorphism of $B$ gives rise to an endomorphism of
$(A_K)^n$, i.e. an
element of $M_n(S)$. Conversely, it is easy to
characterise the elements of $M_n(S)$ that descend to endomorphisms of $B$:

\begin{proposition}[\cite{Howe}, Proposition 2.1]
An element of $M_n(S)$ descends to an endomorphism of $B$ if and only if
it commutes with all elements of the image of $G$ under the cocycle
$a_\Gamma:G\rightarrow \GL_n(R)\leq \GL_n(S)$.
\end{proposition}

Now, we merely need to observe that, as we remarked above, the cocycle
$a_\Gamma$ is in fact nothing but the group homomorphism $G\rightarrow \Aut(\Gamma)$
with respect to an $R$-basis on $\Gamma$. The commutant of its image
in $M_n(S)$ is the intersection of $M_n(S)$ with the commutant
of the image of $a_\Gamma$
in $M_n(D)$, where $D=S\otimes \Q$ is, as in Proposition \ref{prop:main},
assumed to be a field. Moreover, since for any $x\in M_n(D)$, some
integer multiple of $x$ lies in $M_n(S)$, the commutant of $a_\Gamma(G)$ in
$M_n(S)$ is a division ring if and only if its commutant in $M_n(D)$ is
a division algebra.
By Schur's Lemma, the latter is the case if and only if $\Gamma\otimes_R D$
is simple.

Another example in which equation (\ref{eqn}) can be completely analysed is
when $D=S\otimes \Q$ is a quaternion algebra over $F=R\otimes \Q$.
In that case, a theorem of Risman \cite{Risman}
asserts that if $D'$ is any division algebra over $F$, then $D\otimes_F D'$
has zero-divisors if and only if $D'$ contains a splitting field for $D$.
So we immediately deduce:
\begin{proposition}\label{prop:quaternion}
Assume, in addition to Assumption \ref{ass},
that $D$ is a quaternion algebra over $F=R\otimes \Q$. Then
$B$ is simple if and only if $\End_{F[G]}(\Gamma\otimes F)$
contains no splitting field of $D$.
\end{proposition}

A generalisation in a slightly different direction is the special case
that $L=\End_{R[G]}(\Gamma)\otimes \Q$ is a field:
\begin{proposition}\label{prop:gammafield}
Assume, in addition to Assumption \ref{ass},
that $L$ is a field. Suppose also that
$R$ is contained in the centre of $\End(A)$. Then
$B$ is simple if and only if $L$ intersects every splitting
field of $D$ in $F=R\otimes \Q$.
\end{proposition}
\begin{proof}
This follows from the general theory of division algebras,
see e.g. \cite[\S 74A]{CR}. Indeed, let $Z$
be the centre of $D$. If $L\cap Z\neq F$, then certainly $L\otimes_F D$ is
not a division algebra, since $L\otimes_F Z$ is not a field. Suppose that
$L\cap Z=F$, so that $L\otimes_F Z$ is a field.
Then $L\otimes_F D$ is a simple algebra with centre $L\otimes_F Z$.
The dimension of $D$ over $F$ is equal to the dimension of $L\otimes_F D$
over $L$, and their respective dimensions over their centres are therefore also
equal. So $L$ intersects a splitting field of $D$ in a field that is bigger than
$F$ if and only if the index of $L\otimes_F D$ is smaller than that of $D$
if and only if $L\otimes_F D$ has zero divisors.
\end{proof}

\section{Consequences}

We first explain how to deduce Theorem \ref{thm:1} from Propositions
\ref{prop:main} and \ref{prop:quaternion}.

Let $G$ be cyclic of odd prime order $p$.
Recall that $I\leq \Z[G]$ is defined to be the augmentation ideal in $\Z[G]$,
$I=\ker(\sum_{g\in G}n_gg\mapsto \sum_{g\in G}n_g)$. The complexification
$I\otimes \C$ is isomorphic to the direct sum of all non-trivial simple
$\C[G]$-modules, which are all Galois conjugate. It is therefore easy to see
that $I\otimes_\Z \Q$ is a simple $\Q[G]$-module, and that moreover, given any
number field $D$, $I\otimes_\Z D$ is reducible if and only if $D$ intersects
$\Q(\mu_p)$ non-trivially.

First, let $A/k$ be an elliptic curve over a number field.
Then $\End(A)\otimes \Q$ is a field, and the fact that $\End(A) = \End(A_K)$
for an odd degree extension $K/k$ follows from classical CM theory,
see e.g. \cite[Chapter 3]{Lang}. Thus, the dimension 1 case of Theorem
\ref{thm:1} follows from Proposition \ref{prop:main}.

The dimension 2 case is more subtle. Let $A/k$ be an absolutely simple
abelian surface over a number field. Then $\End(A_{\bar{k}})\otimes \Q$ is one
of the following:
\begin{enumerate}
\item $\Q$,
\item a real quadratic number field,
\item a CM field of degree 4,
\item an indefinite quaternion algebra over $\Q$.
\end{enumerate}
We first claim that in all four cases, $\End(A) = \End(A_K)$ for an odd degree
extension $K/k$. This is clear in case 1, and in case 3 this follows from classical
CM theory, see e.g. \cite[Chapter 3]{Lang}. For case 2, observe that the absolute
Galois group of $k$ acts on $\End(A_{\bar{k}})\otimes \Q$ by $\Q$-algebra
automorphisms. If the endomorphism algebra is a quadratic field, then
the action factors through a quotient of $\Gal(\bar{k}/k)$ of index at most 2,
which proves the claim. Finally, case 4 is handled by \cite[Theorem 1.3]{DR}.

If $A/\bar{k}$ is isogenous to a product of elliptic curves, then there are
more possibilities for the structure of $\End(A)$, which have been classified
in \cite[Theorem 4.3]{KedlayaEtAl}. It follows from this classification that
if $\End(A)\otimes \Q$ is a division algebra, then it is still either isomorphic
to $\Q$ or a quadratic field
or a quaternion algebra, and that moreover $\End(A) = \End(A_K)$ for any extension
$K/k$ of degree coprime to 6. So the dimension 2 case of Theorem \ref{thm:1}
follows from Proposition \ref{prop:main} when $\End(A)\otimes \Q$ is a field,
and from Proposition \ref{prop:quaternion} when it is a quaternion algebra,
which covers all possible cases.

To deduce Theorem \ref{thm:2} from Proposition \ref{prop:gammafield},
we use a result of Silverberg, which we will rephrase slightly for our purposes:
for any fixed $d$, there exists a bound $b$ depending only on $d$ (specifically,
$b=4(9d)^{4d}$ is enough),
such that for all abelian varieties over number fields $A/k$ of dimension $d$,
and all extensions $K/k$ of prime degree greater than
$b$, $\End(A) = \End(A_K)$.
Theorem \ref{thm:2} is an immediate consequence of this result together with
Proposition \ref{prop:gammafield}, because $\End_{\Q[G]}(\Gamma\otimes\Q)
\cong \Q(\mu_p)$.

Proposition \ref{prop:main} has an application to questions of simplicity
of Weil restrictions of scalars. If $A/k$ is a simple abelian variety, and $K/k$
is a finite Galois extension with Galois group $G$, then the Weil restriction
of scalars $W_{K/k}(A_K)$ is never simple, since there is a surjective trace map
$W_{K/k}(A_K)\rightarrow A$. Its kernel is, up to isogeny, precisely the
twist $I\otimes_\Z A$, where $I$ is the augmentation ideal in $\Z[G]$. The
following is therefore an immediate consequence of Proposition \ref{prop:main}:

\begin{corollary}
Let $A/k$ be an abelian variety with $\End(A_{\bar{k}})=\Z$. Let $K/k$ be a finite
Galois extension with Galois group $G$. The kernel of the trace map
$W_{K/k}(A_K)\rightarrow A$ is simple over $k$ if and only if $G$ has prime order.
\end{corollary}
\begin{proof}
Cyclic groups of prime order are precisely the finite groups with only two
rational irreducible representations, i.e. those for which $I\otimes_\Z \Q$ is
a simple $\Q[G]$-module.
\end{proof}

If $K/k$ is Galois with dihedral Galois group $G$ of order $D_{2p}$, $p$ an
odd prime, then there is a unique intermediate quadratic extension
$k'=k(\sqrt{d})/k$, and for any abelian variety $A/k$,
$W_{K/k}(A_K)\sim A\times A_d\times X^2$, where $A_d$ is the quadratic
twist of $A$ by $k'/k$. The remaining factor $X$ (up to isogeny) is the
twist of $A$ by a lattice in the $(p-1)$-dimensional irreducible rational
representation $\rho$ of $G$, which is the sum of all the two-dimensional
complex representations of $G$.

\begin{corollary}
Let $E/k$ be an elliptic curve over a number field, $K/k, X$ as above.
Then $X$ is simple.
\end{corollary}
\begin{proof}
The values of each irreducible two-dimensional character of $G$ generate the
maximal real subfield $\Q(\mu_p)^+$ of the $p$-th cyclotomic field,
and they are all Galois conjugate over $\Q$. They will therefore remain
conjugate over any imaginary quadratic field,
so the conclusion holds even when $E$ has CM.
\end{proof}

We conclude with an amusing example of a ``symplectic twist''.
Let $E/k$ be an elliptic curve over a number field,
let $K/k$ be Galois with Galois group $Q_8$, the quaternion group.
There are three intermediate quadratic fields, and correspondingly, the Weil
restriction $W_{K/k}(E_K)$ has, up to isogeny, four factors $E,E_1,E_2,E_3$
that are quadratic twists of $E$. Write
$W_{K/k}(E_K)\sim E\times E_1\times E_2\times E_3\times H$.
\begin{corollary}
Let $K/k$, $E/k$, $H$ be defined as above. Then $H$ is simple, unless
$E$ has CM by an imaginary quadratic field $\Q(\sqrt{-d})$ with
$d$ equal to the sum of three squares, in which case $H$ is isogenous to a
product of two isomorphic simple factors.
\end{corollary}
\begin{proof}
The factor $H$ is (up to isogeny) the twist of $E$ by two copies of the
standard representation of $Q_8$. The endomorphism algebra of this representation
is isomorphic to Hamilton's quaternions, which is split by precisely the
imaginary quadratic fields $\Q(\sqrt{-d})$ for which $d$ is the sum of three
squares.
\end{proof}


\begin{thebibliography}{10}
\bibitem{CR}
C. W. Curtis, I. Reiner, Methods of Representation Theory, with Applications to
Finite Groups and Orders, Vol. 2, John Wiley and Sons, New York, 1987.

\bibitem{DR}
L. V. Dieulefait, V. Rotger,
The arithmetic of QM-abelian surfaces through their Galois representations,
J. Algebra \textbf{281} (2004), 124--143. 

\bibitem{KedlayaEtAl}
F. Fite, K. Kedlaya, V. Rotger, A. Sutherland,
Sato--Tate distributions and Galois endomorphism modules in genus 2,
Compos. Math. \textbf{148}, no. 5 (2012), 1390--1442. 

\bibitem{Howe}
E. Howe, Isogeny classes of abelian varieties with no principal polarizations,
in Moduli of abelian varieties (Texel Island, 1999), Progress in Math.
\textbf{195}, Birkh\"auser, Basel, 2001, 203--216.

\bibitem{Lang}
S. Lang, Complex Multiplication, Grundlehren der mathematischen Wissenschaften
\textbf{255}, Springer Verlag, Berlin, 1983.

\bibitem{MRS}
B. Mazur, K. Rubin, A. Silverberg, Twisting commutative algebraic groups,
J. Algebra \textbf{314} (2007), 419--438.

\bibitem{Ribet}
K. A. Ribet, Endomorphisms of semi-stable abelian varieties over number fields,
Annals Math. \textbf{101} (1975), 555--562.

\bibitem{Risman}
L. J. Risman, Zero divisors in tensor products of division algebras,
Proc. Amer. Math. Soc. \textbf{51} (1975), 35--36.

\bibitem{Serre}
J.-P. Serre, Cohomologie Galoisienne (Cinqui\`eme \'edition, revis\'ee et
compl\'et\'ee), Lecture Notes in Math. \textbf{5}, Springer Verlag, Berlin, 1994.

\bibitem{Silverberg}
A. Silverberg, Fields of definition for homomorphisms of abelian varieties,
J. Pure and Applied Algebra \textbf{77} (1992), 253--262.


\end{thebibliography}
\end{document}